\documentclass[preprint,review,12pt]{elsarticle}

\usepackage{amsmath,amssymb,amsfonts,amsthm}
\usepackage{latexsym,bm}
\usepackage{graphicx}
\usepackage{float}
\usepackage{stmaryrd}

\usepackage{background}
\SetBgContents{}
\usepackage[CJKbookmarks,bookmarks=true,bookmarksnumbered=true]{hyperref}
\hypersetup{pdfsubject={StablizedDGTheory12}, pdfauthor={GZH-CJW of HENU}, colorlinks=true,linkcolor=black,citecolor=green,urlcolor=cyan}

\usepackage{natbib}
\biboptions{numbers,sort&compress}

\begin{document}

\begin{frontmatter}

\title{Analysis of a new stabilized discontinuous Galerkin method for the reaction-diffusion  problem with discontinuous coefficient\tnoteref{1}}\tnotetext[1]{The
work is supported by the Natural Science Foundation of China(No. 10901047).\\
 Email: zhihaoge@henu.edu.cn, fax:+86-378-3881696.}
\author{Zhihao G${\rm e^{1,2}}$,\ Jiwei Ca${\rm o^1}$ }
\address{$ ^1$School of Mathematics and Statistics,
 Henan University, Kaifeng 475004, P.R. China\\
$ ^2$Institute of Applied Mathematics,
 Henan University, Kaifeng 475004, P.R. China}

\begin{abstract}
In this paper, a new stabilized discontinuous Galerkin method within
a new function space setting is introduced, which involves an extra
stabilization term on the normal fluxes across the element interfaces.
It is different from the general DG methods. The formulation satisfies
a local conservation property and we prove well posedness of the new
formulation by Inf-Sup condition. A priori error estimates are derived,
which are verified by a 2D experiment on a reaction-diffusion type model problem.
\end{abstract}

\end{frontmatter}
\thispagestyle{empty}

\newtheorem{thm}{Theorem}[section]
\newtheorem{lem}{Lemma}[section]
\newtheorem{cor}{Corollary}[section]
\newtheorem{prop}{Proposition}[section]
\newtheorem{remark}{Remark}[section]
\renewcommand{\theequation}{\thesection.\arabic{equation}}
\numberwithin{equation}{section}

\section{Introduction}
\setcounter{equation}{0}

In 1973, Reed and Hill \cite{[7]W.H.Reed} introduced the first
discontinuous Galerkin (DG) method for hyperbolic equations,
and since that time there has been an active development of
DG methods for hyperbolic and nearly hyperbolic problems,
resulting in a variety of different methods. Also in the 1970's,
but independently, Galerkin methods for elliptic and parabolic
equations using discontinuous finite elements were proposed by
Babuska, Baker, Douglas and Dupont \cite{[10]J.Douglas}, and a
number of variants introduced and studied such as \cite{[11]D.N.Arnold,[12]D.N.ArnoldandB.Cockburn,[14]F.BrezziandG.Manzini,[15]M.F.Wheeler,[17]P.CastilloandB.Cockburn}  and so on. These were generally
called interior penalty (IP) methods \cite{[8]I.Babuska,[9]I.BabuskaandM.Zlamal} and their development remained
independent of the development of the DG methods for hyperbolic
equations. In 1997, Bassi and Rebay \cite{[19]F.BassiandS.Rebay}
introduced a DG method for the Navier-Stokes equations and in 1998,
Cockburn and Shu \cite{[20]B.CockburnandC.W.Shu} introduced
the local discontinuous Galerkin (LDG) methods. Around the same
time, Oden and Bauman \cite{[21]J.OdenandC.E.Baumann} introduced
another DG method for some diffusion problems. Their approach uses
a non-symmetric bilinear form, even for symmetric problems,
analogous to the one obtained from Nitsche's penalty form by
reversing the sign of the symmetrization term, as discussed earlier.

In this work, we introduce a new stabilized discontinuous Galerkin method within a new function space setting,
which involves an extra stabilization term on the normal fluxes across the element interfaces.
It is different from the general DG methods introduced by
Oden, Babuska and Baumann \cite{[8]I.Babuska, [21]J.OdenandC.E.Baumann}.
The formulation satisfies a local conservation property and we prove well posedness of the
new formulation by proving and using Inf-Sup condition.
A priori error estimates are proved that satisfy the optimal in $h$ and suboptimal
in $p$.

The paper is organized as follows. In Section 2, we introduce the model problem
and the notations. In the same section, the weak formulation of the model problem,
which includes an extra stabilization term on the inter-element jumps of the
fluxes and satisfies a local conservation property, is also described. Then
the discrete problem is given. In Section 3, we investigate the Inf-Sup condition in the
case of discrete. We derive a priori error estimates to analyze the rates of convergence of the
method in Section 4. A 2-dimensional numerical experiments to test our analysis is given
in Section 5. Finally, concluding remarks are summarized in Section 6.

\section{A new stabilized discontinuous Galerkin method }

\subsection{Model problem and notations}
Let $\Omega \subset \mathbb{R}^2$ be a bounded open domain with Lipschitz boundary
$\partial \Omega$ and let $\{P_h\}$ be a family of regular partitions of $\Omega$
into open elements $E$, such that
\begin{eqnarray}
\Omega = \textrm{int} \Bigg( \bigcup\limits_{E \in P_h} \bar{E} \Bigg)
\end{eqnarray}

\begin{figure}[H]
\centering
\includegraphics[width=6.5cm]{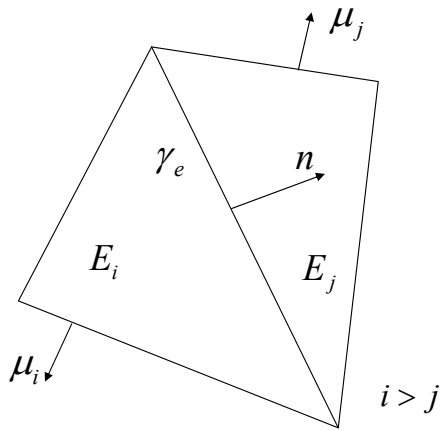}
\caption{Geometrical definitions of neighbouring elements.}\label{figdoc666}
\end{figure}

The following notations will be used in our further considerations. We set
$h_E=\textrm{diam}(E), h=\max_{E \in P_h}h_E$. The set of all edges of the
partition $P_h$ is given by $\varepsilon_h = \{\gamma_h\}, k=1,\ldots,N_{edge}$,
where $N_{edge}$ reprents the number of edges in the partition $P_h$. The interior
interface $\Gamma_{\textrm{int}}$ is then defined as the union of all common edges shared
by elements of partition $P_h$
\begin{eqnarray}
\Gamma_{\textrm{int}} = \bigcup\limits_{k=1}^{N_{edge}}\gamma_e \backslash \partial \Omega.
\end{eqnarray}
The definition of the unit normal vector ${\bm n}$ on each $\gamma_e$ is related to the
numbering of the elements in the partition, such that ${\bm n}$ is defined outward with
respect to the element with the highest index number. The normal vector ${\bm \mu}$ is
defined outward to each element individually. Within the setting above, the following
reaction-diffusion problem is considered
\begin{eqnarray}
- \nabla \cdot (K(x)\nabla u) + u &=& f,\ \textrm{in}\ \Omega,\label{eq3}\\
             u &=& 0,\ \textrm{on}\ \partial \Omega,\label{eq303}
\end{eqnarray}
where $f$ is a real-valued function in $L^2(\Omega)$ and
$0 < K_0 \leq K(x) \leq K_1$. For the sake of clarity
in the notation, the jump and average operators on each
$\gamma_k \in \Gamma_\textrm{int}$ are, respectively, defined as
{\newcommand\langlen{\ensuremath{\langle}}
\begin{eqnarray}\label{eq13}
[v]=v|_{\gamma_e \subset \partial E_i}-v|_{\gamma_e \subset \partial E_j},
\langle v\rangle =\frac12(v|_{\gamma_e \subset \partial E_i}
+ v|_{\gamma_e \subset \partial E_j}),\ i>j,
\end{eqnarray}
where $\gamma_e=\textrm{int}(\partial E_i \cap \partial E_j)$ is
the common edge(in 2D and interface in 3D) between two neighbouring elements.

\subsection{The weak formulation of the problem (\ref{eq3})-(\ref{eq303})}
The discontinuous Galerkin method is a class of finite element methods using
completely discontinuous piecewise polynomial space for the numerical solution
and the test functions. First, we introduce the following $broken\ Sobolev\ space$:
\begin{eqnarray}
\mathcal{M}(P_h)=\{v \in L^2(\Omega)|v\shortmid_E \in H(\Delta,E),\forall E \in P_h,
[\nabla v \cdot {\bm n}] \in L^2(\Gamma_{\textrm{int}})\},\nonumber
\end{eqnarray}
where
\begin{eqnarray}
H(\Delta,E)=\{v \in L^2(E)|\nabla \cdot \nabla v \in L^2(E)\} \subset H^1(E).\nonumber
\end{eqnarray}
The norm of $M(P_h)$ is defined as
{\setlength\arraycolsep{2pt}
\begin{eqnarray}\label{eq1}
|||v|||^2 &=& \sum_{E \in P_h}\Big\{\|v\|^2_{\ast}+\frac{h^\nu}{p^\theta}
\|K(x)\nabla v \cdot {\bm \mu}\|^2_{H^{-1/2}(\partial E)}\Big\} {} \nonumber\\
&& {} +\sigma\frac{h^\lambda}{p^\zeta}\|
[K(x)\nabla v \cdot {\bm \mu}]\|^2_{L^{2}(\Gamma_{\textrm{int}})} {},
\end{eqnarray}}
where $\|v\|^2_{\ast}= \int_E |K(x)| |\nabla v|^2 dx +
\int_E |v|^2dx=\|K^{\frac{1}{2}} \nabla v\|^2_{L^2(E)}+\|v\|^2_{L^2(E)}$.
One can easily prove that norms $\|\cdot\|_{\ast}$ and $\|\cdot\|_{H^1(E)}$
are equivalent. The parameter $p \in \mathbb{R}$ that is introduced here
represents the minimum of all of the local orders of polynomial approximations
$p_E$ in the partition $P_h$. Notice that the parameters $\nu,\lambda,\theta,\zeta$
are greater than or equal to zero and that the subsequent norms in (\ref{eq1})
are defined as
\begin{eqnarray}
\|u\|_{H^{-1/2}(\partial E)}&=&\sup_{\varphi \in H^{1/2}(\partial E)}
\frac{|{\langle u,\varphi \rangle}_{-1/2 \times 1/2,\partial E}|}
{\|\varphi\|_{H^{1/2}(\partial E)}},\label{eq8}\\
\|\varphi\|_{H^{1/2}(\partial E)}
&=&\inf_{\substack{w \in H^{1}(E) \\ \gamma_0 w =\varphi}} \|w\|_{\ast},\label{eq9}
\end{eqnarray}
where ${\langle \cdot,\cdot \rangle}_{-1/2 \times 1/2,\partial E}$ denotes the duality
pairing in $H^{-1/2}(\partial E) \times H^{1/2}(\partial E)$, namely,
\begin{eqnarray}
\langle u, v \rangle_{-1/2 \times 1/2,\partial E}
=\int_{\partial E}uv ds.
\end{eqnarray}
And  $\gamma_0$
denotes the trace operator
$$\gamma_0:H^1(E) \rightarrow H^{1/2}(\partial E).$$
Now, the choice for the space of test functions, $V$, is the
completion of $\mathcal{M}(P_h)$ with respect to the norm $|||\cdot|||$.
The new discontinuous variational formulation, within this
new function space setting, is then stated as follows:
\begin{eqnarray}\label{eq2}
\textrm{Find}\ u \in V, s.t., B(u,v)=L(v),\ \forall v \in V,
\end{eqnarray}
where the bilinear form $B(u,v)$ and linear form $L(v)$ are defined as
\begin{eqnarray}
B(u,v) &=& \sum_{E \in P_h}\Big\{\int_{E}\big(K(x)\nabla u \cdot \nabla v + uv\big)dx {}\nonumber\\
&& {} - \int_{\partial E}\big(v(K(x)\nabla u \cdot {\bm \mu})
      -(K(x)\nabla v \cdot {\bm \mu})u\big)ds\Big\} {} \nonumber\\
&& {} + \int_{\Gamma_{\textrm{int}}}\big(\langle v\rangle[K(x)\nabla u \cdot {\bm n}]
      -\langle u\rangle[K(x)\nabla v \cdot {\bm n}]\big)ds {} \nonumber\\
&& {} + \int_{\Gamma_{\textrm{int}}} \sigma \frac{h^{\lambda}}{p^{\zeta}}
[K(x)\nabla u \cdot {\bm n}][K(x)\nabla v \cdot {\bm n}]ds, {} \label{eq5} \\
L(v) &=& \int_{\Omega}fvdx,\label{eq18}
\end{eqnarray}
where $[\cdot]$ and $\langle \cdot \rangle$ denote the jump
and average operators, respectively.

\begin{remark}\label{rem1201}
The formulation is closely related to the DG formulation by Oden, Babuska
and Baumann \cite{[1]J.T.Oden}. Indeed, if we choose the subspace $\tilde{V}(P_h)$
of $V$ of function with fluxes $\nabla v \cdot {\bm n} \in L^2(\partial E)$,
then we again get the DG formulation of \cite{[1]J.T.Oden}.
The only difference would then be the addition of the last term in (\ref{eq5}).
This term has been incorporated in \cite{[16]P.PercellandM.F.Wheeler,[22]T.J.R.HughesandG.Engel},
where it is accompanied by another penalty term on the jumps of the
function $[v]$ across the element interfaces. We replace the $[v]$ jumps by the
$[\nabla v\cdot {\bm \mu}]$ jumps, in order to prove both continuity and Inf-Sup
properties of the bilinear form with respect to the space $V$, in which
the norm is defined as $|||\cdot|||$.
\end{remark}

\begin{remark}\label{rem1202}
The well posedness of the variational formulation
(\ref{eq2}), and essential in some of these proofs are the continuity
and Inf-Sup conditions of the bilinear form in (\ref{eq5})  are proved in \cite{GeCao}.
\end{remark}

\subsection{The discrete problem}
When implementing the DG methods, we have to compute integrals over volumes(such as
triangles or quadrilaterals in 2D, terahedra or hexahedra in 3D) and faces(such as
edges in 2D, triangles or quadrilaterals in 3D). It would be too costly to compute
the integrals over each physical element in the mesh. A more economical and effective
approach is to use a change of variables to obtain an integral on a fixed element,
called the reference element. Let $\{F_E\}$ be a family of invertible maps defined
for a partition $P_h$ such that every element $E \in P_h$ is the image of $F_E$
acting on a reference element $\hat{E}$, as shown in Figure \ref{fig0011}.
\begin{eqnarray}
F_E: \hat{E} \rightarrow E,\qquad x=F_E(\hat{x}).
\end{eqnarray}

\begin{figure}[H]
\centering
\includegraphics[width=9cm]{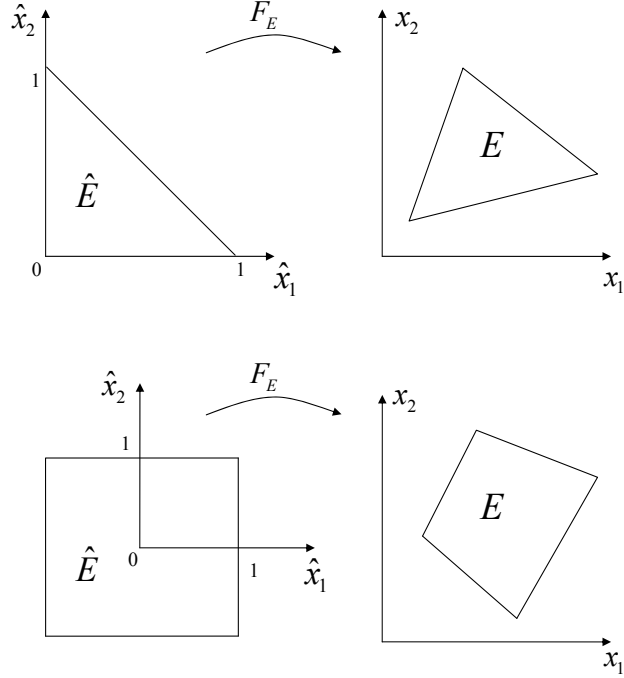}
\caption{Mapping from the reference elements to the physical space.}\label{fig0011}
\end{figure}

In the computational model, a finite-dimensional space of real-valued
piecewise polynomial functions of degree $\leq p_E$ is introduced, such
that
\begin{eqnarray}
V^{hp}=\{v \in L^2(\Omega)|\,v|_E=\hat{v}\circ F^{-1}_E,
\hat{v} \in P^{p_E}(\hat{E}),\forall E \in P_h\}.
\end{eqnarray}
$V^{hp}$ is noted that a subspace of $V$. Now, an  approximation $u_h$ of $u$
is sought as the solution of the following discrete problem:
\begin{eqnarray}\label{eq21}
\textrm{Find}\ u_h \in V^{hp},\ B(u_h,v_h) = L(v_h),\quad v_h \in V^{hp},
\end{eqnarray}
where the bilinear form $B(\cdot,\ \cdot)$ is given by (\ref{eq5}).


\section{Inf-Sup condition on the discrete space $V^{hp}$}

The following trace theorem is used to establish the Inf-Sup condition
on the discrete space $V^{hp}$.
\begin{lem}\label{lem3}
Let $E \in P_h$ be characterized by an affine mapping
$F_{E}$(see section 2.3) and $w \in H^2(E)$. Then there exists a constant
$C>0$, independent of $h_E$, such that
\begin{eqnarray}\label{eq32}
\|\nabla w \cdot {\bm \mu}\|^2_{L^2(\partial E)} \leq C \Big\{
\frac{1}{h_E}\|\nabla w \|^2_{L^2(E)}+\|\nabla w \|_{L^2(E)}
\|\nabla^2 w \|_{L^2(E)}\Big\}.
\end{eqnarray}
\end{lem}
One can find the proof of the lemma in \cite{[5]B.Riviere}.
\begin{cor}\label{cor1}
Given a polynomial(degree $\leq p_E$) $v_h \in P^{p_E}(E)$,
and the mapping between $E$ and $\hat{E}$ is affine, then there exists $C>0$ such that
\begin{eqnarray}
\|\nabla v_h\|^2_{L^2(E)} \geq C\frac{h_E}{p^2_E}
\|\nabla v_h \cdot {\bm \mu}\|^2_{H^{-1/2}(\partial E)}.
\end{eqnarray}
\end{cor}
\begin{proof}
Since $v_h \in P^{p_E}(E)$, we know that $v_h \in H^2(E)$.
Given this information, we can use the following inverse inequality,
obtained from \cite{[27]C.Schwab}.
\begin{eqnarray}
\|\nabla v_h\|_{L^2(E)} \leq C\frac{p^2_E}{h_E}
\|v_h\|_{L^{2}(E)},\quad \forall v_h \in P^{p_E}(E).
\end{eqnarray}
Substituting this inequality into the inequality
(\ref{eq32}) yields
\begin{eqnarray}
\|\nabla v_h \cdot {\bm \mu}\|^2_{H^{-1/2}(\partial E)}
\leq \|\nabla v_h \cdot {\bm \mu}\|^2_{L^{2}(\partial E)}
\leq C\frac{p^2_E}{h_E}\|\nabla v_h\|^2_{L^2(E)}  .
\end{eqnarray}
\end{proof}

\begin{thm}\label{thm4}
Let $\{F_E\}$ be a family of affine invertible mappings. If $\sigma>0$,
then there exists a $\gamma_h = \gamma(\sigma,h,p)>0$, such
that
\begin{eqnarray}
\sup_{v_h \in V^{hp}\backslash\{0\}} \frac{|B(u_h,v_h)|}{|||v_h|||}
\geq \gamma_h |||u_h|||,\qquad \forall u_h \in V^{hp}\backslash\{0\}.
\end{eqnarray}
here we can write
$$\gamma_h = C \min\big\{1,\frac{h^{1-\nu}}{p^{2-\theta}}\big\},\quad C>0.$$
\end{thm}
\begin{proof}
By definition of the supremum, we get
\begin{eqnarray}
\sup_{v_h \in V^{hp}\backslash\{0\}} \frac{|B(u_h,v_h)|}{|||v_h|||}
\geq \frac{B(u_h,u_h)}{|||u_h|||}.
\end{eqnarray}
By applying Corollary \ref{cor1}, it is clear that there exists
$C>0$, such that
\begin{eqnarray}
B(u_h,u_h) &\geq& C\Big\{\sum_{E \in P_h}\Big(\|u_h\|^2_{\ast}+\frac{h}{p^2}
\|K(x)\nabla u_h \cdot {\bm n}\|^2_{H^{-1/2}(\partial E)}\Big) {} \nonumber\\
&& {} +\sigma\frac{h^\lambda}{p^\zeta}\|
[K(x)\nabla u_h \cdot {\bm n}]\|^2_{L^{2}(\Gamma_{\textrm{int}})} {}\Big\},
\end{eqnarray}
which completes the proof.
\end{proof}

\section{Error estimation}

In this section, we investigate the convergence properties of solution
$\{u_h\}$ of (\ref{eq21}). Let $u \in V$ be the exact solution to the
VBVP (\ref{eq2}), then by using the linearity of $B(\cdot,\cdot)$, it
follows easily that the approximation error $e_h=u-u_h$ is governed by
\begin{eqnarray}
B(e_h,v)=L(v)-B(u_h,v)=R^h(v),\quad \forall v \in V,
\end{eqnarray}
where $R^h:V \rightarrow \mathbb{R}$ is the residual functional. Note
that, due to (\ref{eq21}), the residual satisfies the following orthogonality
property on $V^{hp}$,
\begin{eqnarray}\label{eq33}
B(e_h,v_h)=0,\quad \forall v_h \in V.\label{eq23}
\end{eqnarray}

We start firstly by proving an interpolation theorem in the norm
$|||\cdot|||$, secondly we need for our proof of the error estimate
in this norm. Then we derives the convergence rates in the norm
$\|\cdot\|_{H^1(P_h)}$. Finally we prove an error estimate in the
$L^2$ norm.


We introduce a family of interpolants $\{\pi^{E}_{hp}\}$, such that
\begin{eqnarray*}
&&\pi^{E}_{hp}: H^{\gamma_{k}}(E) \rightarrow P^{p_{E}}(E),\\
&&\pi^{E}_{hp}(v_h) = v_h,\quad \forall v_h \in P^{p_{E}}(E).
\end{eqnarray*}

Some results for the interpolation errors from the work of Babuska
and Suri \cite{[23]I.BabuskaandM.Suri} will be needed.
\begin{thm}\label{thm3}
For $\varphi \in H^{\gamma_{k}}(E)$, there
exists $C>0$, independent of $\varphi,p_k$ and $\gamma_k$ and a
sequence $\pi^{E}_{hp} \in P^{p_{E}}(E)$, such that
\begin{eqnarray*}
&&\|\varphi - \pi^{E}_{hp}(\varphi)\|_{L^2(E)}
\leq C\frac{h^{\mu_k}}{p^{r_k}_{E}}\|\varphi\|_{H^{\gamma_k}(E)},\\
&&\|\nabla (\varphi - \pi^{E}_{hp}(\varphi))\|_{L^2(E)}
\leq C\frac{h^{\mu_k-1}}{p^{r_k-1}_{E}}\|\varphi\|_{H^{\gamma_k}(E)},\\
&&\|\nabla^2 (\varphi - \pi^{E}_{hp}(\varphi))\|_{L^2(E)}
\leq C\frac{h^{\mu_k-2}}{p^{r_k-2}_{E}}\|\varphi\|_{H^{\gamma_k}(E)},
\end{eqnarray*}
where $\mu_k = \min\{p_E+1,\gamma_k\},\ \gamma_k \geq 1,\ p_E \geq 1$.
\end{thm}

By extending the local interpolant operators $\pi^{E}_{hp}(\cdot)$ to zero outside
of $E$ for every $E \in P_h$, we can define the global interpolant
$\Pi_{hp}$ as follows
\begin{eqnarray}\label{eq22}
\Pi_{hp}: V \rightarrow V^{hp},\quad \Pi_{hp}(u)
=\sum_{E \in P_h}\pi^{E}_{hp}(u|_E),\quad u \in V.
\end{eqnarray}

We define the notations as follows,
\begin{eqnarray*}
e_h &=& u-u_h = u-\Pi_{hp}(u) + \Pi_{hp}(u)-u_h\\
&=& \eta_h + \xi_h,
\end{eqnarray*}
where $\eta_h=u-\Pi_{hp}(u),\ \xi_h=\Pi_{hp}(u)-u_h.$

\begin{thm}\label{thm6}
Let $u \in V$ and $u|_E \in {H^{\gamma_k}}(E), \gamma_k \geq 2$,
let the stabilization parameters $\sigma>0$ and $\lambda,\zeta
\geq 0$, and let the norm parameters $\nu,\theta \geq 0$, there exists $C>0$,
independent of $u,h$, and $p$ such that the interpolation error $\eta_h = u-\Pi_{hp}(u)$
can be bounded as follows
\begin{eqnarray*}
|||\eta_h||| \leq C \frac{h^{\mu^{\ast}}}{p^{\gamma^{\ast}}}
\|u\|_{H^{\gamma_k}(P_h)},\quad \gamma_k \geq 2,
\end{eqnarray*}
where
\begin{eqnarray*}
&&\mu^{\ast}=\min \big\{\mu-1,\mu-\frac32+\frac\nu2,
\mu-\frac32+\frac\lambda 2 \big\},\\
&&\gamma^{\ast}=\min \big\{\gamma-1,\gamma-\frac32+\frac\theta 2,
\gamma-\frac32+\frac\zeta 2 \big\},\\
&&\mu=\min \big\{p+1,\gamma\big\},
\end{eqnarray*}
and where $\gamma = min_{E \in P_h}(\gamma_k)$.
\end{thm}
The norm in the space $H^{\gamma_k}(P_h)$ is defined as
$$\|u\|_{H^{\gamma_k}(P_h)}=
\sqrt{\sum_{E \in P_h} \|u\|^2_{H^{\gamma_k}(E)}}.$$

\begin{proof}
By recalling the definition the norm $|||\cdot|||$ (\ref{eq1}), applying
the triangle inequality, and using Lemma \ref{lem3}, we get
\begin{eqnarray*}
|||\eta_h|||^2 &\leq& C(\sigma) \sum_{E \in P_h} \Big\{\|\eta_h\|^2_{L^2(E)}
+\Big(1+\frac{h^{\nu-1}}{p^\theta}+\frac{h^{\lambda-1}}{p^\zeta}\Big)
\|\nabla \eta_h\|^2_{L^2(E)}{}\\
&& +\Big(\frac{h^{\nu}}{p^\theta}+\frac{h^{\lambda}}{p^\zeta}\Big)
\|\nabla \eta_h\|_{L^2(E)}\|\nabla^2 \eta_h\|_{L^2(E)}\Big\}.{}
\end{eqnarray*}
Now, application of the interpolation Theorem \ref{thm3} gives our
final result
\begin{eqnarray*}
|||\eta_h|||^2 \leq C(\sigma) \Big(\frac{h^{2\mu-2}}{p^{2\gamma-2}}
+\frac{h^{2\mu-3+\nu}}{p^{2\gamma-3+\theta}}
+\frac{h^{2\mu-3+\lambda}}{p^{2\gamma-4+\zeta}}\Big)
\sum_{E \in P_h}\|u\|^2_{H^{\gamma_k}(E)}.
\end{eqnarray*}
\end{proof}

\begin{thm}\label{thm9}
Given $\sigma >0$, let $u \in H^2(\Omega) \cap V$ be the unique
solution to the variational BVP (\ref{eq2}), $u_h \in V^{hp}$
be an approximation (\ref{eq21}) of $u$, and both the stabilization
and norm parameters be of order $O(h/p^2)$(i.e.,$\lambda = \nu
=1$ and $\zeta = \theta =2$). Then, the error $u-u_h$ satisfies
the bound
\begin{eqnarray*}
|||u-u_h||| \leq C(\sigma) \frac{h^{\mu-3/2}}{p^{\gamma-2}}
\|u\|_{H^{\gamma_k}(P_h)},\quad p \geq 1,\ \gamma \geq 2,
\end{eqnarray*}
where  $\gamma = min_{E \in P_h}(\gamma_k)$,
$\mu = \min\{p+1,\gamma\}$ and $C(\sigma)$ is a
positive constant depending on $\sigma$.
\end{thm}

\begin{proof}
Given the interpolation $\Pi_{hp}(u)$ defined in (\ref{eq22}).
Note that $\xi_h \in V^{hp}$ and that the interpolation error
$\eta_h \in H^2(P_h)$. Consequently, using the triangle inequality,
we obtain
\begin{eqnarray}
|||u-u_h||| \leq |||\eta_h||| + |||\xi_h|||.\label{eq24}
\end{eqnarray}
Applying the discrete Inf-Sup condition of Theorem \ref{thm4}
and taking $\nu = 1$ and $\theta =2$, we obtain
\begin{eqnarray*}
|||\xi_h||| \leq C\sup_{v_h \in V^{hp}\backslash\{0\}}
\frac{|B(\xi_h,v_h)|}{|||v_h|||}.
\end{eqnarray*}
Using the orthogonality property (\ref{eq23}), the
inequality can be rewritten as
\begin{eqnarray*}
|||\xi_h||| \leq C\sup_{v_h \in V^{hp}\backslash\{0\}}
\frac{|B(\eta_h,v_h)|}{|||v_h|||}.
\end{eqnarray*}
Taking $\lambda =1$ and $\zeta =2$, we have
\begin{eqnarray*}
|||\xi_h||| \leq C(\sigma)\frac{p}{\sqrt{h}}|||\eta_h|||.
\end{eqnarray*}
Thus, returning to (\ref{eq24}), we can conclude
\begin{eqnarray*}
|||u-u_h||| \leq C(\sigma)\frac{p}{\sqrt{h}}|||\eta_h|||.
\end{eqnarray*}
With this choice of coefficients, we get for
the parameters in Theorem \ref{thm6}
$$\mu^{\ast} = \mu -1,\quad \gamma^{\ast} = \gamma - 1.$$
Hence, we finish the proof.
\end{proof}
\begin{remark}
Since $\|v\|_{H^1(P_h)} \leq |||v|||$, the above theorem
would imply suboptimal convergence rates for the error
in the norm $\|\cdot\|_{H^1(P_h)}$.
\end{remark}


In the following, we derive optimal $h$ convergence rates for
the error in the norm $\|\cdot\|_{H^1(P_h)}$ for $p \geq 2$.
The decisive element in the proof is the use of a specific
type of interpolants, that were introduced by Riviere et al
\cite{[24]B.RiviereandM.F.Wheeler}.
The original estimates were them improved in
\cite{[28]B.RiviereandM.F.Wheeler, [29]S.PrudhommeandF.Pascal}.
By using these interpolants, we succeed at proving optimal $h$
convergence rates for $p \geq 2$, but the $p$ convergence
rates appear to be $1/2$ order lower than the one predicted
in the previous section. So we succeed at improving the
$h$ convergence rate, but not the $p$ rates.

We start our analysis by defining the following norm on
$H^2(P_h)$,
\begin{eqnarray}
|||v|||^2_{H^2(P_h)}=\sum_{E \in P_h}
\Big\{\|v\|^2_{\ast}+\frac{p^\zeta}{\sigma h^\lambda}
\|v\|^2_{L^2(\partial E)}
+\frac{\sigma h^\lambda}{p^\zeta}
\|\nabla v \cdot {\bm \mu}\|^2_{L^2(\partial E)}\Big\}.
\end{eqnarray}

\begin{lem}\label{lem4}
The bilinear form $B$ is continuous on $H^2(P_h) \times
H^2(P_h)$ with respect to the norm $\|\cdot\|_{H^1(P_h)}$
, i.e., there exists $N>0$, such that
\begin{eqnarray}
|B(u,v)| \leq N|||u|||_{H^2(P_h)} |||v|||_{H^2(P_h)},
\quad \forall u,v \in H^2(P_h),
\end{eqnarray}
where $N$ is a constant, independent of $h$ and $p$.
\end{lem}

Next, we introduce an important inverse inequality (see
\cite{[2]A.Romkes}) between the spaces $H^1(P_h)$ and $H^2(P_h)$
for finite-dimensional functions $v_h \in V^{hp}$.
\begin{lem}\label{lem5}
Let the parameters in the norm $\|\cdot\|_{H^1(P_h)}$ be set
as $\lambda=\zeta=1$, then there exists a constant $C$
dependent of $\sigma$, such that
\begin{eqnarray}
|||v_h-\bar{v_h}|||_{H^2(P_h)} \leq C(\sigma)
\sqrt{p}\|v_h\|_{H^1(P_h)},\quad \forall v_h \in V^{hp},
\end{eqnarray}
where $\bar{v_h}$ denotes the piecewise average of $v_h$
\begin{eqnarray}\label{eq27}
\bar{v_h} = \sum_{E \in P_h} \bar{v_h}|_E,
\quad \bar{v_h}|_E = \frac{1}{|E|}\int_E v_h dx.
\end{eqnarray}
\end{lem}

Similar to our proofs in the other sections, we need a theorem
on the interpolation error in the norm $\|\cdot\|_{H^2(P_h)}$.
As mentioned previously, here we do not use the Babuska and
Suri \cite{[23]I.BabuskaandM.Suri} interpolants but rather
the interpolants proposed by Riviere et al \cite{[24]B.RiviereandM.F.Wheeler}.

\begin{thm}\label{thm7}
Let $\varphi \in H^{\gamma_k}(E)$,
$\gamma_k \geq 2$, there exists $C >0$, independent of
$\varphi,p_k$ and $\gamma_k$, and interpolant
$\tilde{\pi}^{E}_{hp}(\varphi) \in P^{p_k}$, such that
\begin{eqnarray}\label{eq28}
\int_{\gamma \subset \partial E}\nabla
(\varphi - \tilde{\pi}^{E}_{hp}(\varphi)) \cdot {\bm \mu} ds = 0
\end{eqnarray}
and
\begin{eqnarray*}
&&\|\varphi - \tilde{\pi}^{E}_{hp}(\varphi)\|_{L^2(E)}
\leq C\frac{h^{\mu_k}}{p^{r_k-3/2}_{E}}\|\varphi\|_{H^{\gamma_k}(E)},\\
&&\|\nabla (\varphi - \tilde{\pi}^{E}_{hp}(\varphi))\|_{L^2(E)}
\leq C\frac{h^{\mu_k-1}}{p^{r_k-3/2}_{E}}\|\varphi\|_{H^{\gamma_k}(E)},\\
&&\|\nabla^2 (\varphi - \tilde{\pi}^{E}_{hp}(\varphi))\|_{L^2(E)}
\leq C\frac{h^{\mu_k-2}}{p^{r_k-2}_{E}}\|\varphi\|_{H^{\gamma_k}(E)},
\end{eqnarray*}
where $\mu_k = \min\{p_E+1,\gamma_k\},\ p_E \geq 2$.
\end{thm}

Again, by extending the corresponding local interpolant
$\tilde{\pi}^{E}_{hp}$ equal to zero outside of each
$E \in P_h$, we can define a global interpolant on $V$
\begin{eqnarray}\label{eq25}
\tilde{\Pi}_{hp}:V \rightarrow V^{hp},\quad \tilde{\Pi}_{hp}(u)
= \sum_{E \in P_h} \tilde{\pi}^{E}_{hp}(u|_E),\quad u \in V.
\end{eqnarray}

\begin{thm}\label{thm8}
Let $u \in H^2(\Omega) \cap V,
\tilde{\Pi}^{E}_{hp}(u) \in V^{hp}$ be the interpolant of
$u$ (\ref{eq25}) and let the stabilization parameters
$\sigma >0$ and $\lambda ,\zeta \geq 0$, then there
exists $C(\sigma)>0$, independent of $u,h$, and $p$ such
that the interpolation error can be bounded as follows
\begin{eqnarray*}
|||u-\tilde{\Pi}_{hp}(u)|||_{H^2(P_h)} \leq C(\sigma)
\frac{h^{\mu^{\ast \ast}}}{p^{\gamma^{\ast \ast}}}
\|u\|_{H^{\gamma_k}(P_h)},
\quad \gamma_k \geq 2,\ p \geq 2,
\end{eqnarray*}
where
\begin{eqnarray*}
&&\mu^{\ast \ast}=\min \big\{\mu-1,\mu-\frac12-\frac{\lambda}2,
\mu-\frac32+\frac\lambda 2 \big\},\\
&&\gamma^{\ast \ast}=\min \big\{\gamma-\frac32,\gamma-\frac32-\frac{\theta}2,
\gamma-\frac74+\frac{\zeta}2 \big\},\\
&&\mu=\min \big\{p+1,\gamma\big\},
\end{eqnarray*}
and where $\gamma=\min_{E \in P_h}(\gamma_k)$.
\end{thm}

\begin{proof}
Recalling the definition of the norm $|||\cdot|||_{H^2(P_h)}$,
substituting trace inequality and Lemma \ref{lem3}, we can get
\begin{eqnarray*}
|||\eta_h|||^2_{H^2(P_h)} &\leq&  \sum_{E \in P_h} \Big\{\|\eta_h\|^2_{\ast}
+\frac{h^{-(\lambda+1)}}{\sigma p^{-\zeta}}\Big(\|\eta_h\|^2_{L^2(E)}
+h\|\nabla \eta_h\|^2_{L^2(E)}\|\eta_h\|_{L^2(E)}\Big){}\\
&& +\frac{\sigma h^{\lambda-1}}{p^\zeta}
\Big(\|\nabla \eta_h\|^2_{L^2(E)}+h\|\nabla \eta_h\|_{L^2(E)}
\|\nabla^2 \eta_h\|_{L^2(E)}\Big)\Big\}.{}
\end{eqnarray*}
Applying  Theorem \ref{thm7}, we complete
the proof.
\end{proof}

\begin{thm}\label{thm10}
Given $\sigma >0$. Let $u \in H^2(\Omega) \cap V$
be the exact solution to the VBVP (\ref{eq2}), $u_h \in V^{hp}$ be
a discrete approximation (\ref{eq21}) and let the stabilization
parameter be of order $O(h/p)$(i.e.,$\lambda =1$ and $\zeta =1$),
then there exists $C(\sigma) \geq 0$ such that
\begin{eqnarray*}
\|u-u_h\|_{H^1(P_h)} \leq C(\sigma) \frac{h^{\mu-1}}{p^{\gamma-5/2}}
\|u\|_{H^{\gamma_k}(P_h)},\quad p \geq 2,
\end{eqnarray*}
where $\mu = \min \{p+1,\gamma\}$.
\end{thm}
\begin{proof}
Given the interpolant $\tilde{\Pi}_{hp}(u)$ in (\ref{eq25}),
using the triangle inequality, we can obtain
\begin{eqnarray}\label{eq30}
\|u-u_h\|_{H^1(P_h)} \leq \|\eta_h\|_{H^1(P_h)}+\|\xi_h\|_{H^1(P_h)}.
\end{eqnarray}
From (\ref{eq5}) and (\ref{eq18}), it follows that
\begin{eqnarray}
\|\xi_h\|^2_{H^1(P_h)} \leq B(\xi_h,\xi_h).
\end{eqnarray}
Using the orthogonality property (\ref{eq23}) and the linearity
of $B(\cdot,\cdot)$, this can be rewritten as
\begin{eqnarray}
\|\xi_h\|^2_{H^1(P_h)} \leq B(\eta_h,\xi_h)=B(\eta_h,\xi_h-\bar{\xi_h})
+B(\eta_h,\bar{\xi_h}),
\end{eqnarray}
where $\bar{\xi_h}$ denotes the piecewise average (\ref{eq27})
of $\xi_h$. Now, applying Lemma \ref{lem4} to the first
term in the right hand side, we get
\begin{eqnarray}
\|\xi_h\|^2_{H^1(P_h)} \leq C|||\xi_h-\bar{\xi_h}|||_{H^2(P_h)}
\|\eta_h\|_{H^2(P_h)}+B(\eta_h,\bar{\xi_h}).
\end{eqnarray}
Applying the inverse inequality of Lemma \ref{lem5},
we can rewrite the above inequality as
\begin{eqnarray}\label{eq29}
\|\xi_h\|^2_{H^1(P_h)} \leq C(\sigma)\sqrt{p}\|\xi_h\|_{H^1(P_h)}
|||\eta_h|||_{H^2(P_h)}+B(\eta_h,\bar{\xi_h}).
\end{eqnarray}
As we shall now see, the term $B(\eta_h,\bar{\xi_h})$ can be
bounded in terms of $\|\xi_h\|_{H^1(P_h)}$ as well, due to
the special property (\ref{eq28}) of the interpolant
$\tilde{\Pi}_{hp}$. By expanding the term $B(\eta_h,\bar{\xi_h})$,
we get
\begin{eqnarray*}
B(\eta_h,\bar{\xi_h})=\sum_{E \in P_h}\Big(\int_E \eta \bar{\xi_h}dx
+\int_{\partial E}\bar{\xi_h}\nabla \eta_h \cdot {\bm \mu} ds\Big)
-\int_{\Gamma_\textrm{int}} \langle \bar{\xi_h} \rangle
[\nabla \eta_h \cdot {\bm n}]ds.
\end{eqnarray*}
Now, applying the property (\ref{eq28}), gives
\begin{eqnarray*}
B(\eta_h,\bar{\xi_h})=\sum_{E \in P_h} \int_E \eta_h \bar{\xi_h}dx
\leq \|\eta_h\|_{L^2(\Omega)}\|\bar{\xi_h}\|_{L^2(\Omega)}
\leq C\|\eta_h\|_{L^2(\Omega)}\|\xi_h\|_{H^1(P_h)}.
\end{eqnarray*}
Back substitution of this result into (\ref{eq29}) and
(\ref{eq30}), then yields
\begin{eqnarray*}
\|u-u_h\|_{H^1(P_h)} \leq C(\sigma)\sqrt{p}
|||\eta_h|||_{H^2(P_h)}.
\end{eqnarray*}
Next, recalling the interpolation Theorem \ref{thm8}, we get
\begin{eqnarray*}
\|u-u_h\|_{H^1(P_h)} \leq C(\sigma)
\frac{h^{\mu^{\ast \ast}}}{p^{\gamma^{\ast \ast}-1/2}}
\sqrt{\sum_{E \in P_h} \|u\|^2_{H^{\gamma_k}(E)}}.
\end{eqnarray*}
Since $\lambda = \zeta =1$, we know that $\mu^{\ast \ast}
=\mu -1$ and $\gamma^{\ast \ast} = \gamma -2$.
\end{proof}
\begin{remark}
If we combine the results of Theorem \ref{thm9} and \ref{thm10},
we can conclude that for a stabilization term of order $O(h/p^2)$
and for $p \geq 2$,  the convergence rates are of order $\mu -1$
and $\gamma -2$ for $h$ and $p$ convergence, respectively.
\end{remark}

Now, we prove the error estimate in the $L^2$ norm. We will apply
the Aubin-Nitsche lift technique used in the analysis of the
classical finite element method to the DG method. First we
introduce an important result (see \cite{[5]B.Riviere}) as follows.
\begin{thm}\label{thm11}
Let $v \in H^{\gamma_k}(E)$ for $\gamma_k > 1$. Let $P_E \geq 0$
be an integer. There exists a constant $C>0$ independent of
$v$ and $h_E$ and a function $\tilde{v} \in P^{p_E}(E)$ such that
for all $0 \leq s \leq \gamma_k$,
\begin{eqnarray}
\|v-\tilde{v}\|_{H^s(E)} \leq Ch_E^{\min\{p_E+1,\gamma_k\}-s}
\|v\|_{H^{\gamma_k}(E)}.
\end{eqnarray}
\end{thm}

We assume that the domain is convex and that the solution to
the dual problem
\begin{eqnarray*}
- \nabla \cdot (K(x)\nabla \phi) + \phi &=& e_h,\ {\rm in}\ \Omega,\\
\phi &=& 0,\ {\rm on}\ \partial \Omega
\end{eqnarray*}
belongs to $H^2(\Omega)$ with continuous dependence on $e_h=u-u_h$,
\begin{eqnarray}\label{eq34}
\|\phi\|_{H^2(\Omega)} \leq C\|e_h\|_{L^2(\Omega)}.
\end{eqnarray}
Then, we have
\begin{eqnarray*}
&&\quad \|e_h\|^2_{L^2(\Omega)}=\int_{\Omega}(e_h)^2dx
=\int_{\Omega}(- \nabla \cdot (K(x)\nabla \phi) + \phi)e_hdx\\
&&=\sum_{E \in P_h}\int_E(K(x)\nabla\phi \cdot \nabla e_h+\phi e_h)dx
-\sum_{E \in P_h}\int_{\partial E}(K(x)\nabla \phi \cdot {\bm \mu})e_hds\\
&&=\sum_{E \in P_h}\int_E (K(x)\nabla\phi \cdot \nabla e_h+\phi e_h)dx
-\int_{\Gamma_{\rm int}}\langle K(x)\nabla \phi \cdot {\bm \mu}\rangle [e_h]ds,
\end{eqnarray*}
because of the regularity of $\phi$ we know that the jumps
$[K(x)\nabla \phi \cdot {\bm \mu}]|_{\gamma_e}=0$.
Now subtracting the orthogonality property (\ref{eq33}) from the
equation above, we get
\begin{eqnarray}
\|e_h\|^2_{L^2(\Omega)}&=&\sum_{E \in P_h}\int_E
(K(x)\nabla(\phi-v_h) \cdot \nabla e_h+ (\phi-v_h)e_h)dx{}\nonumber\\
&&{}-\int_{\Gamma_{\rm int}}\langle K(x)\nabla(\phi-v_h) \cdot {\bm \mu} \rangle [e_h]ds
-\int_{\Gamma_{\rm int}}\langle K(x)\nabla e_h \cdot {\bm n}\rangle [v_h]ds{}\nonumber\\
&&{}-\int_{\Gamma_{\rm int}} \sigma\frac{h^\lambda}{p^\zeta}
[K(x)\nabla e_h\cdot {\bm n}][K(x)\nabla v_h\cdot {\bm n}]ds{}\label{eq36}\\
&=&A_1 + A_2.\nonumber
\end{eqnarray}
We choose $v_h = \tilde{\phi}$, a continusous interpolant of $\phi$
of degree $p_E$, and assume that such an interpolant exists. In this
case, the first term is easily bounded using Cauchy-Schwarz
inequality and the Theorem \ref{thm11},
\begin{eqnarray*}
A_1&=&\sum_{E \in P_h}\int_E
(K(x)\nabla(\phi-\tilde{\phi})\cdot \nabla e_h
+ (\phi-\tilde{\phi})e_h)dx\\
&\leq& C\|\phi\|_{H^1(\Omega)} |||e_h|||
\leq Ch\|\phi\|_{H^2(\Omega)} |||e_h|||.
\end{eqnarray*}
The rest terms except $A_1$ in the right hand side (\ref{eq36}), yields
\begin{eqnarray*}
|A_2| \leq |B(e_h,\phi-\tilde{\phi})|
\leq M |||e_h|||\,|||\phi-\tilde{\phi}|||.
\end{eqnarray*}
Therefore, by the Theorem \ref{thm10} and using
the bound (\ref{eq34}), we obtain
\begin{eqnarray*}
\|e_h\|^2_{L^2(\Omega)} \leq C \frac{h^{\mu}}{p^{\gamma-5/2}}
\|e_h\|_{L^2(\Omega)} \|u\|_{H^{\gamma_k}(P_h)}.
\end{eqnarray*}
Then, we have the following theorem:
\begin{thm}\label{thm12}
Assume that Theorem \ref{thm10} holds. Then there exists a constant
$C(\sigma)$ independent of $h$ and $p$, but dependent of $\sigma$,
such that
\begin{eqnarray*}
\|u-u_h\|_{L^2(\Omega)} \leq C(\sigma) \frac{h^{\mu}}{p^{\gamma-5/2}}
\|u\|_{H^{\gamma_k}(P_h)},\quad p \geq 2,
\end{eqnarray*}
where $\mu = \min \{p+1,\gamma\}$.
\end{thm}

\begin{remark}
For the results of Theorem \ref{thm12}, we can find that for the
stabilization term of order $O(h/p^2)$ and for $p \geq2$, the
convergence rates are of order $\mu$ and $\gamma-2$ for $h$ and
$p$ convergence, respectively.
\end{remark}

\section{Numerical results}
For the 2D example problem, we consider the following VBVP, given on
the unit square $\Omega = (0,1)\times(0,1)$ with prescribed Dirichlet
boundary conditions on $\partial \Omega$,
\begin{eqnarray}
\begin{split}
-\nabla \cdot (K(x,y) \nabla u) + u &= f,\ \textrm{in}\ \Omega,\\
u &= 0,\ \textrm{on}\ \partial \Omega.
\end{split}
\end{eqnarray}
Here we take $K(x,y) = xy$, and the exact solution to this problem we
choose is
\begin{eqnarray}\label{eq35}
u(x,y) = xy(1-x)(1-y).
\end{eqnarray}
In the convergence analyses performed here, the orders of the norm
and stabilization parameters are set equal to $O(h)$, i.e.,
taking $\lambda=\nu=1$. In Figure \ref{fig111}, the results are shown
for the approximation error in the $L^2(\Omega)$ norm.
Figure \ref{fig222} shows the convergence rates with respect to the $L^2(\Omega)$ norm. And the rule of obtaining the rates is defined as
\begin{eqnarray}
\beta_h = \frac{\log(e^i_h/e^{i+1}_h)}{\log 2}.
\end{eqnarray}
Figure \ref{fig333} and \ref{fig444} show the exact solution (\ref{eq35})
and its contour figure, respectively. At last, the numerical solution is
shown by using our method in Figure \ref{fig555}.

\begin{figure}[H]
\centering
\includegraphics[width=10cm]{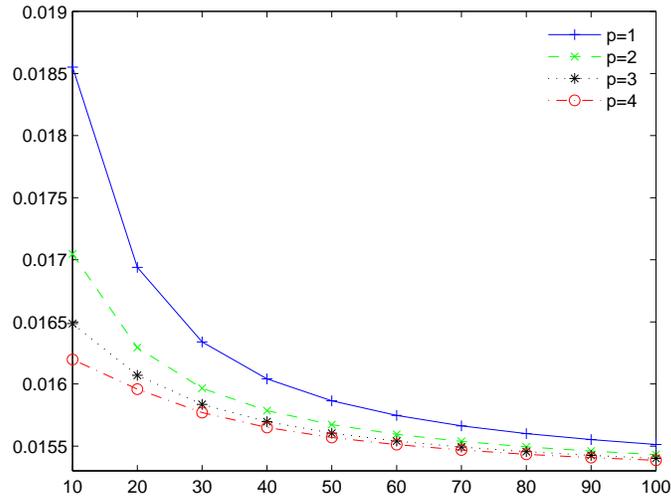}
\caption{$L^2$ norm error for p taking 1,2,3 and 4.}\label{fig111}
\end{figure}
\begin{figure}[H]
\centering
\includegraphics[width=10cm]{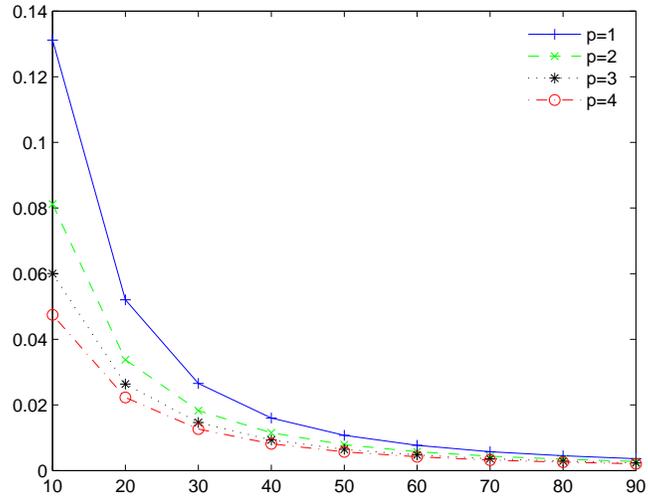}
\caption{Convergence rate $L^2$ norm for p taking 1,2,3 and 4.}\label{fig222}
\end{figure}
\begin{figure}[H]
\centering
\includegraphics[width=10cm]{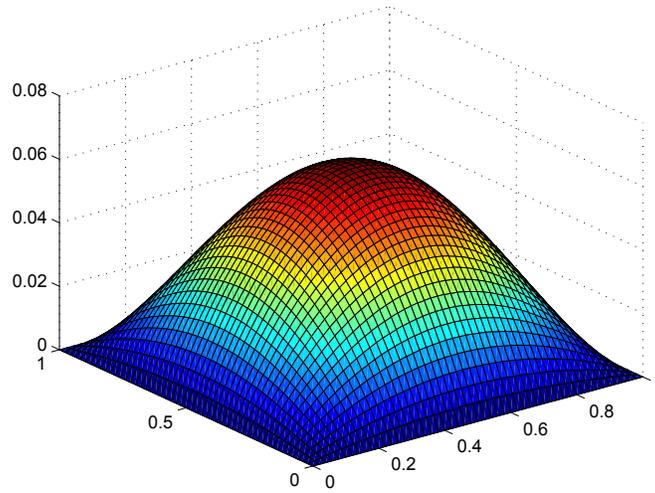}
\caption{Exact solution.}\label{fig333}
\end{figure}
\begin{figure}[H]
\centering
\includegraphics[width=10cm]{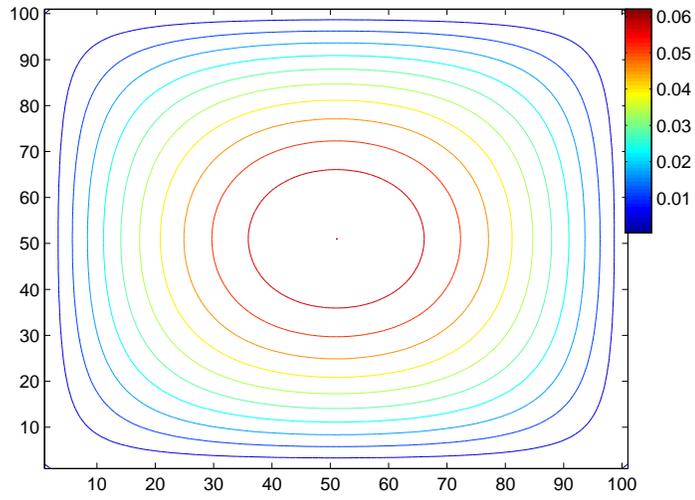}
\caption{Contour figure of exact solution.}\label{fig444}
\end{figure}
\begin{figure}[H]
\centering
\includegraphics[width=9cm]{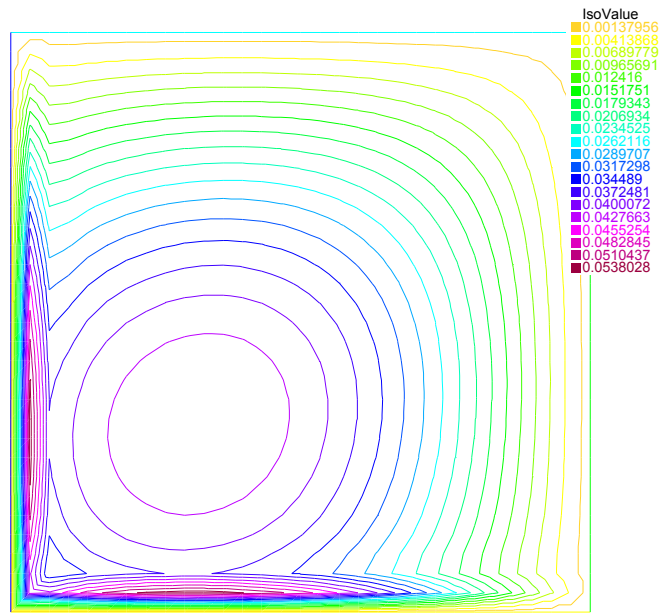}
\caption{Contour figure of DG solution.}\label{fig555}
\end{figure}

\section{Conclusion}
We introduce a new DG formulation and analyse the case of
a two-dimensional reaction-diffusion problem with Dirichlet
boundary conditions. The method is similar to the general
DG method of \cite{[13]I.BabuskaandC.Baumann, [1]J.T.Oden},
but involves an extra stabilization term on the jumps of the
fluxes across the element interfaces. In the work, we apply the
conforming mesh as Figure \ref{figdoc666} shows, but generally
we can constrct elements as the Figure \ref{fig000}
shows, and the elements are even star-shaped as $E_i$. Appliations
of this type of element one can find in \cite{[3]V.Dolejsi}.
\begin{figure}[H]
\centering
\includegraphics[width=8cm]{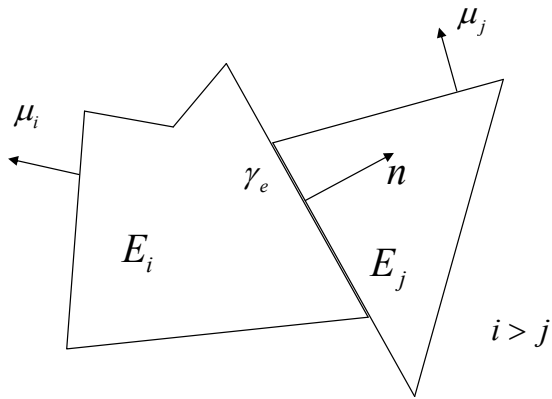}
\caption{More general geometrical elements.}\label{fig000}
\end{figure}

In addation, a new space setting is introduced. Instead of choosing the conventional
$H^2(P_h)$, which is predominantly used in discontinuous Galerkin methods, we relax the constrains on the space and choose functions that are locally in $H(\Delta,K)$ and whose jumps in the fluxes
across the element interfaces are in $L^2(\Gamma_{\rm int})$.


\end{document}